\documentclass[12pt]{amsart}
\usepackage{amscd}
\usepackage{verbatim}
\usepackage{amssymb, amsmath, amsthm, amscd,ifthen}
\usepackage[dvips]{graphics}
\usepackage[cp866]{inputenc}
\usepackage{graphicx}
\usepackage{epsfig}
\usepackage{xcolor}

\usepackage{amsmath,amssymb,amscd,}
\usepackage[all]{xy}

\usepackage{amssymb}

\theoremstyle{definition}

\newtheorem{theorem}{Theorem}

\newtheorem{lemma}[theorem]{Lemma}
\newtheorem{proposition}[theorem]{Proposition}
\newtheorem{definition}{Definition}

\DeclareMathOperator{\R}{\mathbb{R}}
\DeclareMathOperator{\eps}{\varepsilon}

\title{A  note on the concurrent normal conjecture}

\author[A. Grebennikov]{Alexandr Grebennikov}
\author[G. Panina]{Gaiane Panina}

\address[A. Grebennikov]{Saint-Petersburg State University, Department of Mathematics and Computer Sciences}
\email[]{sagresash@yandex.ru}
\address[G. Panina]
{ St.\ Petersburg Department of Steklov Mathematical Institute}
\email[]{gaiane-panina@rambler.ru}


\begin{document}

\begin{abstract}
 It is conjectured since long that for any
convex body  $K \in \mathbb{R}^n$ there exists a point in
the interior of $K$ which belongs to at least $2n$ normals from different points on the
boundary of $K$. The conjecture is known to be true for $n=2,3,4$.

Motivated by a recent preprint of Y. Martinez-Maure, we give a short proof  of his result:   \textit{for dimension $n\geq 3$, under mild conditions, almost every normal
through a boundary point to a smooth convex body $K\in \mathbb{R}^n$ 
contains an intersection point  of at least $6$ normals from different points on the
boundary of $K$.} 
\end{abstract}

 \maketitle \setcounter{section}{0}

\section{Introduction}\label{sec:Overview}
Given a smooth convex body $K \in \mathbb{R}^n$, its normal to a point $p\in \partial K$ is a line passing through $p$ and orthogonal to $\partial K$ at the point $p$.
 It is conjectured  that for any
convex body  $K \in \mathbb{R}^n$ there exists a point in
the interior of $K$ which is the intersection point  of at least $2n$ normals from different points on the
boundary of $K$.
The concurrent normals conjecture  trivially holds  for $n=2$. For $n=3$ it was proven by Heil \cite{H1} and \cite{H2}  via geometrical methods and reproved by Pardon via topological methods.
The case $n=4$ was completed also by Pardon \cite{P}.

Recently Martinez-Maure proved for  $n=3,4$   that (under some mild conditions) almost every normal
through a boundary point to a smooth convex body $K$
 passes arbitrarily close to the set
of points lying on normals through at least six distinct
points of $\partial K$ \cite{M-M}.

He used  Minkowski differences of smooth convex bodies, that is, the \textit{theory of hedgehogs}.

The present paper  is very much  motivated by  \cite{M-M}. We give an alternative short proof of almost the same fact for all $n\geq 3$, see Theorem \ref{ThmMain}. Our proof is based on the bifurcation theory  and does not use hedgehogs.

\bigskip

\textbf{Acknowledgments.} G. Panina is supported by RFBR grant 20-01-00070A; A. Grebennikov is supported by Ministry of Science and Higher Education of the Russian Federation, agreement 075-15-2019-1619.

\newpage

\section{Some preliminaries and the main result}\label{prelim}

Let $n \ge 3$, and let $K$ be a strictly convex $C^{\infty}$-smooth compact body in $\R^n$. For a point $x \in \partial K$  we denote by $\mathcal{N}(x)$ the normal line to $\partial K$ at the point $x$.

Following \cite{M-M}, we make use of the support function. \footnote{Equivalently, one can work with the squared distance function to the boundary.}

 Let $h: S^{n-1} \to \R$   be the support function of $K$. For $y \in \mathbb{R}^n$ define $$h_y: S^{n-1} \to \R$$
$$
    h_y(v) = h_0(v) - \langle v, y\rangle.
$$

The function $h_y$  equals the support function of $K$ after a translation which takes the origin $O$  to the point $y$.

Given a point $y$, all the normals passing through $y$ can be read off  the function $h_y$:

\begin{lemma}\label{Lemma1}\cite{M-M}
 A   point $y$ lies on the normal  $\mathcal{N}(x)$ for some $x\in \partial K$
  iff 
   $u(x)$  is a critical point of the function $h_y$.
   
   Here $u(x)\in S^{n-1}$  is the outer unit normal to $\partial K$ at the point $x$.
   \qed
\end{lemma}

By Morse lemma type arguments (\cite{Morse}, Lemma A), $h_y$ is a Morse function almost for all $y$.
Its bifurcation diagram \cite{A}, \cite{M-M} is given by the\textit{ focal surface} $\mathcal{F}_K$ of the body $K$.
\begin{enumerate}
  \item  The focal surface
equals the locus of the centers of principal curvatures  of $\partial K$. Thus it  has $n-1$ \textit{sheets}. Sheet  number $k$ corresponds
to  the $k$-th curvature, assuming that the curvature radii are enumerated in the ascending order:
$r_1\leq r_2 \leq...\leq r_{n-1}.$  Each sheet is an image of $S^{n-1}$. The sheets intersect each other and may have self-intersections and singularities.

\item The focal surface cuts the ambient space $\R^n$ into \textit{cameras} (that is, connected components of the complement of $\mathcal{F}_K$).
The type of the associated Morse functions $h_y$ depends on the camera containing $y$ only. 
\item Transversal crossing  of exactly one of the sheets  (say, the sheet number $k$) of the focal surface at its smooth point amounts to a birth (or death) of two critical points of $h_y$ whose indices are $k$ and $k-1$.
\end{enumerate}

\begin{theorem}\label{ThmMain}
Let $n\geq3$, and let $K\in \mathbb{R}^n$ be a $C^{\infty}$-smooth convex body, $x \in \partial K$. If the normal line $\mathcal{N}(x)$ does not intersect the singular locus of the focal surface $\mathcal{F}_K$
  then  $\mathcal{N}(x)$
contains  a point $z$ such that:

\begin{enumerate}
  \item $z$ is an intersection point  of at least $6$ normals from different points on the
boundary of $K$.
  \item The distance $|xz| $  satisfies $$r_1(x)<|xz|<r_{n-1}(x),$$  where $r_1(x)$ and $r_{n-1}(x)$ are  the largest and the smallest principal curvature radii  at the point $x$.
\end{enumerate}
\end{theorem}


\section{Proof  of Theorem \ref{ThmMain}}\label{prelim}

\begin{definition}
    A one-parametric family  $f_t \in C^{\infty}(S^{n-1}, \R), \ \ t \in \R_{+}$ is \textit{nice} if the following properties hold:
    \begin{enumerate}
        \item $f_t$ depends  smoothly on $t$;
        \item $f_t$ is a Morse function for each $t$ except for finitely many \textit{bifurcation points} $t_1, \ldots t_m$;
        \item  each of the bifurcation points is of one of the two types:
          \begin{enumerate}
                                                            \item A\textit{ birth/death point}. Some two critical points with some neighbor indices $k$ and $k-1$ collide and disappear, or, vice versa,  two critical points with neighbor indices appear.
                                                            \item An \textit{index exchange point}. Two critical points with neighbor indices $k$ and $k-1$ collid
                                                           at $t=t_i$. There arises a degenerate critical point  which splits afterwards into two critical points with the same  indices $k$ and $k-1$.
                                                          \end{enumerate}

    \end{enumerate}
    \end{definition}

    For a nice family, denote $T = \R_{+} \setminus \{ t_1, \ldots, t_m\}$. For each $t \in T$ and $0 \le k \le n-1$ let $C_{t, k}$  be the number of critical points of $f_t$ of index $k$, and let  $N_t$  be the total number of critical points of the function $f_t$.

\begin{lemma}
    \label{discrete-continuity}
    Let $f_t$ be a {nice} family of functions. Let also $C_{t_1, 0} \ge 2$, $C_{t_2, n-1} \ge 2$ and
    $$
        \sum_{k = 1}^{n-2} C_{t, k} \ge 1 \text{ for all } t \in T \cap [t_1, t_2].
    $$
    Then there exists $t \in T \cap [t_1, t_2]$ such that $N_t \ge 6$.
\end{lemma}
\begin{proof} Assume the contrary, that is, for every $t$ we have $N_t < 6$.
    By assumption of the lemma, there are other critical points of $f_t$ than max and min, so $N_t$ is at least $3$  for all $t$. Besides, $N_t$ is even, so $N_t$  necessarily equals $4$ for all  $t \in T \cap[t_1,t_2]$.

    Therefore there are no birth/death bifurcations on $[t_1, t_2]$.  Together with the conditions on $C_{t_1, 0}$ and $C_{t_2, n-1}$, this implies that $f_t:S^n \rightarrow \mathbb{R}$ has two maxima and two minima for any $t\in T\cap [t_1, t_2]$ and no other critical points. A contradiction.
\end{proof}

Now we are ready to prove Theorem \ref{ThmMain}. Assume that the point $x$ is such that the normal $\mathcal{N}(x)$ does not meet the singularity locus of the focal surface. Denote   by $u=u(x)$ the outer unit normal to $\partial K$ at the point $x$.
  Then the    principal curvature radii $r_1 < r_2 < \ldots < r_{n-1}$ of $\partial K$ at the point $x$    are all different, and the family of functions $\{h_{x - tu}\}_{t \in \R_{+}}$ is {nice}. The set $T$ of bifurcation points  includes $\{r_1,...r_{n-1}\}$.  Let us prove that there exists $r \in (r_1, r_{n-1})$, such that $x - ru$ lies on $\ge 6$ normals.

 By Lemma \ref{Lemma1}, $u$ is a critical point of $h_{x - tu}$ for all $t$. It is easy to check that its Morse index equals $0$ when $t \in (0, r_1)$, equals $1$ when $t \in (r_1, r_2)$, \ldots, equals $n-1$ when $t \in (r_{n-1}, +\infty)$ \cite{M-M}.

    We shall apply Lemma \ref{discrete-continuity} for the family $\{ h_{x - tu} \}$ on the segment $[t_1 ,t_2] := [r_1 + \eps, r_{n-1} - \eps]$ for sufficiently small $\eps > 0$. Since the Morse index of the point $u$ is neither $0$, nor $n-1$  on this segment, the condition $\sum_{k = 1}^{n-2} C_{t, k} \ge 1$ is satisfied. So, we only need to check that $C_{t_1, 0} \ge 2, C_{t_2, n-1} \ge 2$.

    We will check the first inequality, the second may be verified in a similar way.   The point $u$ at $t=r_1$ is an index exchange point. This means that  as $t$ equals $r_1+\varepsilon$  and tends to $r_1$, there is a (local) minimum point $u_{\varepsilon}$ of $h_{x - (r_1+\varepsilon)u}$ which tends to the critical point $u$ of index $1$. We conclude that for small $\varepsilon$, the point $u_{\varepsilon}$ is not the global minimum of $h_{x - (r_1+\varepsilon)u}$. Therefore there are two local minima, that is, two distinct critical points of index $0$.

    Application of Lemma \ref{discrete-continuity} completes the proof.\qed

\end{document}